\documentclass[letterpaper, 11pt]{article}

\usepackage{hyperref}
\usepackage{amscd,amssymb,amsfonts,amsmath,amsthm}

\newtheorem{theo}{Theorem}
\newtheorem{lem}[theo]{Lemma}

\newtheorem{prop}[theo]{Proposition}

\theoremstyle{definition}
\newtheorem{defi}[theo]{Definition}

\newtheorem{remark}[theo]{Remark}

\newcommand{\E}{\mathcal E}
\newcommand{\F}{\mathcal F}
\def\L{\mathcal{L\/}}
\def\hra{\hookrightarrow}
\def\deg{\operatorname{deg}}

\begin{document}

\title{An elementary proof that vector bundles on $\mathbb P^1$ split}
\author{William F. Sawin}

\maketitle

\begin{abstract}This paper gives a new elementary proof of the theorem that all vector bundles on $\mathbb P^1$ split into the direct sum of line bundles. The proof is based on the study of divisors associated to germs of sections at the generic point.

\end{abstract} 

Dedekind and Weber first proved that all vector bundles on $\mathbb P^1$ split into the direct sum of line bundles, in terms of the factorization of matrices, in \cite{DedekindWeber82}.  Grothendieck later proved it using a cohomology argument in \cite{Grothendieck57}. This paper presents a new elementary proof, based on studying divisors of meromorphic sections. We will use elementary standard facts from algebraic geometry which can be found in \cite[Chapter II]{Hartshorne77}.

The writing of this paper was aided by a tremendous number of comments from Tobias Dyckerhoff and Mikhail Kapranov.

\vspace{10pt}

Let $X$ be a nonsingular irreducible algebraic curve, and $\E$ a vector bundle on $X$ (i.e. a locally free sheaf of finite rank). Let $K$ be the field of rational functions on $X$. If we localize $\E$ at the generic point $\xi$, we get a $K$-vector space, $\E_\xi$. We begin by attaching some additional structure to that vector space.

\begin{defi}Let $P$ be a point on $X$ and $\alpha$ an element of $\E_\xi$. Let $t$ be a local parameter at $P$. Since there is a injection from $\E_P$ to $\E_\xi$, we can view the former as an $\mathcal O_P$-submodule of the latter. We define the {\em order of $\alpha$ at $P$} to be

\[ \delta_P(\alpha)=\max \{ n | t^{-n} \alpha \in \E_P\}.\]

This number does not depend on our choice of $t$.\end{defi}

\begin{defi}

We define the {\em divisor $\delta(\alpha)$ of $\alpha$} to be the divisor whose value at each $P$ is $\delta_P(\alpha)$. The {\em degree $\deg(\alpha)$ of $\alpha$} is the degree of $\delta(\alpha)$.

\end{defi}

It is easy to check that $\delta_P(\alpha)$ is always finite, and that $\delta(\alpha)$ has finite support. The proofs are mostly analogous to the proofs for the divisor function on $K$, which we will also refer to as $\delta$.

\begin{lem}\label{properties} The divisor function $\delta$ has the following properties.

\begin{enumerate}

\item For $f\in K$, $\alpha\in \E_\xi$, we have

\[ \delta(f \alpha)=\delta(f)+\delta(\alpha).\]

\item For $\alpha_1,...\alpha_n\in \E_\xi$, we have

\[\delta\left(\sum_{i=1}^n\alpha_i\right)\geq \min_i \delta(\alpha_i).\]

\end{enumerate}

\end{lem}

\begin{proof} Analogous to the corresponding properties of the usual divisor function on $K$.\end{proof}

\begin{remark}Each point on a curve corresponds to a valuation, which can be converted into an absolute value by the formula $|f|=q^{-\delta_P(f)}$ for $q>1$, the same process used to create p-adic absolute values. The corresponding function $q^{-\delta_P(\alpha)}$ on $\E_\xi$ gives it the structure of a Banach space. This sort of space is discussed in \cite{GoldmanIwahori63}.\end{remark}

\begin{defi}Recall that when we identify a vector bundle with its sheaf of sections, then line bundles correspond to locally free sheaves of rank $1$. Given an injective map $\F \hra \E$ of locally free sheaves, we call $\F$ a subbundle of $\E$ if for every point $f: {P} \to X$ the map on fibers $f^*\F \to f^*\E$ is injective.\end{defi}

\begin{prop}\label{bundles}

\begin{enumerate} The following hold.

\item Let $\F$ and $\E$ be vector bundles. The isomorphisms $\E \stackrel{\sim}{\to} \F$ are in one-to-one correspondence with those isomorphisms $\E_{\xi} \stackrel{\sim}{\to}\F_{\xi}$ which commute with $\delta$.

\item For a vector bundle $\E$, every $1$-dimensional subspace $L \subset \E_{\xi}$ corresponds to a unique sub line bundle $\L \hra \E$ such that $L = \L_{\xi}$.

\item Let $\{e_i\}$ be a basis of $\E_\xi$. $\E$ splits into the line bundles corresponding to the spaces generated by the $e_i$ if and only if, for all $f_1,...,f_n \in K$

\begin{equation}\label{eq.split} \delta(\sum_{i=1}^{n} f_i e_i) = \min_i \delta(f_i e_i). \end{equation}
\end{enumerate}

\end{prop}

\begin{proof}

To prove part 1, we will reconstruct the modules and maps that make up the sheaf $\E$ from $\delta$. First note that $\delta_P(\alpha)\geq 0$ if and only if $\alpha \in \E_P$. This means that $\alpha \in \E(U)$ if and only if for all $ P\in U$,  $\delta_P(\alpha)\geq 0$. Therefore we can reconstruct $\E(U)$ by taking the set of all elements in $\E_\xi$ with that property. Since each section is identified with its germ at $\xi$, the restriction maps are given by the natural inclusion maps. Since this construction is unique, the result follows.

To prove part 2, we use the argument from part 1 to see how the subspace $L$ corresponds to a subsheaf $\L \hra \E$. Fix a nonzero $\alpha \in L$. It is clear from part 1 of Lemma \ref{properties} that $\L$ is the sheaf of sections of the line bundle corresponding to the divisor $\delta(\alpha)$. It remains to verify that $\L \hra \E$ is a subbundle. This follows since $\L$ is locally generated by the nonvanishing section $t^{-\delta_P(\alpha)} \alpha$.

To prove part 3, first assume that $\E$ splits into a direct sum of line bundles. Then the divisor function on $\E_{\xi}$ satisfies (\ref{eq.split}). To obtain the other direction, note that, by part 1, property (\ref{eq.split}), combined with the divisors of the line bundles, characterizes the sheaf $\E$ up to isomorphism. \end{proof}

\begin{lem}\label{technical} Let $\alpha_1,...\alpha_n \in \E_\xi$ and suppose

\[\delta_P\left(\sum_{i=1}^n \alpha_i\right) > \min_i (\delta_P(\alpha_i)).\]

\begin{enumerate}

\item  Let $J$ be the set of $i$ such that $\delta_P(\alpha_i)=\min_i (\delta_P(\alpha_i))$. Then

\[\delta_P\left(\sum_{i \in J} \alpha_i\right) > \min_{i\in J} (\delta_P(\alpha_i)).\]

\item Assume $X=\mathbb P^1$. Further, assume that for all $i$, we have $ \delta_P(\alpha_i)=\min_i (\delta_P(\alpha_i))$.

Let $\alpha_j$ be an element of the smallest degree among $\{\alpha_i\}$. Then we can choose $f_1,...,f_n\in K$ such that

\[\delta\left(\sum_{i=1}^n f_i \alpha_i\right)> \delta(\alpha_j).\]

\end{enumerate}

\end{lem}

\begin{proof}

To prove part 1

\[\delta_P\left(\sum_{i \in J} \alpha_i\right) =\delta_P\left(\sum_{i=1}^n \alpha_i - \sum_{i\not\in J} \alpha_i\right)\geq \min\left( \delta_P\left(\sum_{i=1}^n \alpha_i\right), \min_{i\not\in J} \delta_P(\alpha_i)\right)\]

and since both $\delta_P\left(\sum_{i=1}^n \alpha_i\right)$ and each $\delta_P(\alpha_i)$ with $i\not\in J$ are greater than $\min_i \delta_P(\alpha_i)$, their sum is as well.

To prove part 2, for each $i$ we choose $D \ge \delta(\alpha_j)$ with $\deg(D) = \delta(\alpha_i)$ and $\delta_P(D) = \delta_P(\alpha_j)$. Since by assumption $X = {\mathbb P}^1$, there exists $f_i \in K$ such that $\delta(f_i\alpha_i) = D$ and we may further assume that $f_i(P) = 1$.

Since $\delta(f_i \alpha_i)\geq \delta(\alpha_j)$, we know that

\[\delta\left(\sum_{i=1}^n f_i\alpha_i \right)\geq\delta(\alpha_j).\]

To prove strict inequality, we will show that the order of the divisor of the sum is strictly larger at $P$. Intuitively this is because multiplying by $f_i$ did not change the value of the sections at $P$, so they still cancel. Algebraically, we have

\[\delta_P\left(\sum_{i=1}^n f_i \alpha_i \right)=\delta_P\left(\sum_{i=1}^n (f_i-1)\alpha_i+\sum_{i=1}^n \alpha_i\right)>\delta_P(\alpha_j)\] \end {proof}

The two parts of Lemma \ref{technical} naturally combine, as the sum produced by the first satisfies the conditions required by the second. We can now prove our theorem.

\begin{theo} All vector bundles on $\mathbb P^1$ split into a direct sum of line bundles. \end{theo}

\begin{proof} Consider a vector bundle $\E$. First we show that there must be a maximal degree among the divisors of the sections. If $\E_\xi$ contains a section of degree $d$, then $\E$ contains the line bundle generated by that section, which has degree $d$ and so has a $d+1$-dimensional vector space of global sections. All these global sections will also be global sections of the vector bundle. Since it has only a finite-dimensional vector space of global sections, there must be a maximum degree.

Choose a section $e_1\in\E_\xi$ of highest degree. For each $i$ choose an $e_i$ of highest degree among sections not in the span of $e_1$ through $e_{i-1}$. Continue until the whole space is spanned. The set $\{e_1,..,e_n\}$ clearly form a basis of $\E_\xi$.

We will now argue that, for all $f_1,...,f_n\in K$, we have

\[ \delta\left(\sum_{i=1}^n f_ie_i\right)=\min_i \delta(f_ie_i)\]

By Proposition \ref{bundles}, this proves our claim. Assume that this equation fails for some $\{f_i\}$. We know that the left hand side cannot be less at any point, so it must be greater somewhere.

Choose a $P$ where equality does not hold. We use the first part of Lemma \ref{technical}, then the second, to choose a new weighted sum $\sum g_i e_i$ of the $e_i$. Because some of the elements are removed by the first part, some of the $g_i$ will be $0$. Let $e_j$ be the last, and therefore lowest-degree, basis element with nonzero coefficient $g_j$. The sum will satisfy

\[\delta\left( \sum_{i=1}^n g_i e_i\right)>\delta( f_je_j)\]

As $g_j\neq 0$, $\sum_{i=1}^ng_ie_i$ is not in the span of $e_1$ through $e_{j-1}$, but its degree is greater than $e_j$'s. This is a contradiction.\end{proof}

\bibliographystyle{alpha}

\end{document}